\documentclass{amsart}

\usepackage{amsmath, amssymb, amsthm, amsfonts}

\usepackage[utf8]{inputenc}

\input xy
\xyoption{all}

\usepackage{hyperref}

\usepackage{color}
\usepackage{isomath}

\swapnumbers
\numberwithin{equation}{section}

\theoremstyle{plain}
\newtheorem{theorem}[subsection]{Theorem}
\newtheorem{lemma}[subsection]{Lemma}
\newtheorem{prop}[subsection]{Proposition}
\newtheorem{cor}[subsection]{Corollary}

\theoremstyle{definition}

\newtheorem{remark}[subsection]{Remark}

\def\AA{\mathbb{A}}

\def\CC{\mathbb{C}}

\def\GG{\mathbb{G}}

\def\bJ{\mathbb{J}}

\def\PP{\mathbb{P}}

\def\RR{\mathbb{R}}

\def\ZZ{\mathbb{Z}}

\def\calD{\mathcal{D}}
\def\calE{\mathcal{E}}

\def\calH{\mathcal{H}}

\def\calL{\mathcal{L}}
\def\calM{\mathcal{M}}

\def\calP{\mathcal{P}}

\newcommand\cE{\mathcal{E}}

\def\bJ{\mathbf{J}}

\newcommand{\gr}{\textup{gr}}

\newcommand\rk{\textup{rk}}

\newcommand\Sp{\textup{Sp}}

\newcommand{\Gm}{\GG_m}

\newcommand\quash[1]{}

\renewcommand\a\alpha
\renewcommand\b\beta
\newcommand\g\gamma
\renewcommand\d\delta

\newcommand{\El}{\textup{El}}
\newcommand{\MC}{\textup{MC}}
\newcommand{\prim}{\textup{prim}}

\newcommand{\para}{{\rm par}}

\newcommand{\specialcell}[2][c]{%
  \begin{tabular}[#1]{@{}c@{}}#2\end{tabular}}

\title{Irregular Hodge numbers for Rigid $G_2$-Connections}

\author{Konstantin Jakob}
\thanks{K.J. is supported by the DFG Research Fellowship JA 2967/1-1.}
\address{(K.J.) Massachusetts Institute of Technology, Department of Mathematics, 77 Massachusetts Ave, MA 02139, USA}
\email{kjakob@mit.edu}
\author{Stefan Reiter}
\address{(S.R.) Universität Bayreuth, Mathematisches Institut, Lehrstuhl für Zahlentheorie, 95440 Bayreuth, Germany}
\email{stefan.reiter@uni-bayreuth.de}

\begin{document}

\begin{abstract} Certain rigid irregular $G_2$-connections constructed by the first-named author are related via pullbacks along a finite covering and Fourier transform to rigid local systems on a punctured projective line. This kind of property was first observed by Katz for hypergeometric connections and used by Sabbah and Yu to compute irregular Hodge filtrations for hypergeometric connections. This strategy can also be applied to the aforementioned $G_2$-connections and we compute jumping indices and dimensions for their irregular Hodge filtrations.
\end{abstract}
 
\maketitle

\section{Introduction}
Initiated by Deligne \cite{De07}, the goal of irregular Hodge theory is to provide an analogue of the theory of variations of Hodge structures in the context of irregular singular differential equations.
These equations are of interest in several areas of mathematics ranging from mirror symmetry to the geometric Langlands program. \par

In analogy to the case of rigid local systems for which Simpson proves in \cite[Corollary 8.1]{Sim90} that they underlie a complex variation of Hodge structure (provided their local monodromy has eigenvalues with absolute value one), 
Sabbah proves in \cite[Theorem 0.7]{Sa18} that an irreducible rigid irregular connection can be equipped with a canonical irregular Hodge filtration (if the eigenvalues of the formal monodromy at every singularity have absolute value one). \par

Construction and classification of $G_2$-connections goes back to the work of Dettweiler and the second-named author \cite{DR10} who classified tamely ramified rigid $G_2$-local systems. The classification is carried out explicitly and relies on the work of Katz on rigid local systems \cite{Ka96}. He defines an operation called middle convolution and proves that any rigid local system may be constructed from a local system of rank one using twists with rank one local systems and middle convolution. The classification of \cite{DR10} lead in particular to the construction of a family of motives for motivated cycles with motivic Galois group $G_2$. \par
The work of Katz was generalised by Arinkin \cite{Arinkin10} (and Deligne) who proved that any rigid irregular connection may be constructed from a rank one connection if one allows the additional operation of Fourier transform. Using these results the classification of tamely ramified rigid $G_2$-connections was generalised by the first-named author in \cite{Ja20} to irregular $G_2$-connections with slope of the form $1/k$ for some positive integer $k$. \par
In the case of rigid local systems the work of Dettweiler and Sabbah \cite{DS13} provides an algorithmic way to compute the Hodge data of a rigid local system provided one knows how to construct it from rank one using middle convolution. Unfortunately, in general in the irregular case it is still unknown how the irregular Hodge filtration changes under Fourier transform. Therefore one has to make use of other tools to compute irregular Hodge filtrations for rigid connections. \par
First explicit results for Hodge filtrations of hypergeometric connections have been obtained by Sevenheck and Domínguez in \cite{CDS19}. 
These results were extended by Sabbah and Yu to arbitrary non-resonant hypergeometric connections in \cite{SY19}, using a result of Katz \cite[Theorem 6.2.1.]{Ka90} which relates irregular hypergeometric connections to regular singular hypergeometric connections via pullbacks and Fourier transform. Additionally in \cite{FSY18} irregular Hodge numbers for symmetric powers of Kloosterman connections are computed. Apart from these results explicit computations of irregular Hodge filtrations remain rare. \par
The goal of this article is to prove that a similar stability property holds for non-hypergeometric rigid irregular $G_2$-connections constructed in \cite{Ja20} and to compute their irregular Hodge filtrations and expand the list of computable examples of irregular Hodge filtrations. \par

\subsection{Results}
Consider the rigid irregular $G_2$-connections with local data at $0$ and $\infty$ given by
\begin{center}
\begin{tabular}{ c c }
$0$ & $\infty$ \\
\hline \\
$(\bJ(3),\bJ(3),1)$ & \specialcell[c]{$\El( 2,1,(\alpha,\alpha^{-1}))$ \\ $\oplus\, \El(2,2,1) \oplus (-1)$} \\ [15pt]
$(-\bJ(2),-\bJ(2),E_3) $ & \specialcell[c]{$\El( 2,1,(\alpha,\a^{-1}))$ \\ $\oplus\, \El(2,2,1)\oplus (-1)$} \\ [15pt]
$(-\beta E_2,-\beta^{-1}E_2,E_3)$ & \specialcell[c]{$\El( 2,1,(\a,\a^{-1}))$ \\ $\oplus\, \El(2,2,1)\oplus (-1)$} \\
\hline \\
$(\bJ(3),\bJ(2), \bJ(2))$ & \specialcell[c]{$\El(2,1,1) \oplus \El(2,t,1)$ \\ $\oplus\,\El(2,t+1,1) \oplus (-1)$} 
\end{tabular}
\end{center}
with $t \in \CC^\times$, $\alpha=\exp(-2\pi i a), \beta=\exp(-2\pi i b)$ and $a,b\in (1/2,1)$, constructed in \cite[Theorem 1.1.]{Ja20}.  Here by $x\bJ(s)$ we denote a Jordan block of size $s$ with eigenvalue $x$ and 
\[\El(2,2,1)=\El(u^2,2/u,1)\]
is an elementary module in the sense of \cite{Sa08}, Section 2. \par
Denote the above connections by $\calE_1, \calE_2, \calE_3$ and $\calE_4$, numbered from top to bottom. 
By \cite[Theorem 0.7]{Sa18} these connections are equipped with a canonical irregular Hodge filtration $F_\textup{irr}^\beta$. This is a decreasing filtration (indexed by real numbers) of the fiber
\[H_i=\calE_i|_{z=1} \]
which is a finite dimensional $\CC$-vector space. It is only well-defined up to a global shift of the index. We denote by 
\[\gr_{F_{\textup{irr}}}^\beta(H_i):=F_{\textup{irr}}^\beta / \sum_{\gamma > \beta} F_{\textup{irr}}^\gamma \]
the associated graded. The set of indices $\beta$ for which $\gr_{F_{\textup{irr}}}^\beta(H_i) \neq 0$ is finite and these indices are called the jumping indices (or jumps) of the filtration. Our main result is the following.
\begin{theorem} \label{thm: irreg hodge} Denote by $F_{\textup{irr}}$ the irregular Hodge filtration and let $d_\alpha^p(-)=\dim \gr_{F_\textup{irr}}^{\alpha+p}(-)$.  We have the following jumping indices and irregular Hodge numbers (up to a global shift) for $\calE_1$, $\calE_2$ and $\calE_4$.
 \[  
 \begin{array}{cc}
     p &   d_{1/2}^p\\
      \hline
     0 & 2  \\
     1 & 3  \\
     2 & 2
    \end{array} \]

For $\calE_3$ we have the following cases.

\[    \begin{array}{ccc}
     p     & d^p_{2b} &d^p_{1}\\
      \hline
     0 & 2  & 0 \\
     1 & 2 & 3 
    \end{array},\quad 3/4=b< a, \] 
    \[     \begin{array}{cccccc}
     p    &  d^p_{2b}&  d^p_{1} \\
      \hline  
     0 & 2  & 1 & \\
     1 & 2  & 1 &  \\
     2 & 0  & 1
    \end{array},\quad 3/4=b > a  
\]   

\[     \begin{array}{ccccc}
     p   &  d^p_{2b} &  d^p_{2(1-b)+1} &  d^p_{1} \\
      \hline  
     0 & 1 &1 & 1\\
     1 & 1 & 1 & 1 \\
     2 & 0 & 0 & 1
    \end{array},\quad b > a, b\neq 3/4
\]
\[    \begin{array}{cccc}
     p &  d^p_{2b} &  d^p_{2(1-b)+1}&d^p_{1}\\
      \hline
     0 & 2 & 0 & 0 \\
     1 & 0 & 2 & 3 \\
 
    \end{array},\quad 1/2<b< a, b\neq 3/4 \]

\end{theorem}
The connections above are not all that appear in \cite[Theorem 1.1.]{Ja20}. There are six more rigid $G_2$-connections with the following local data
\begin{center}
\begin{tabular}{ c c }
$(iE_2,-iE_2,-E_2,1)$ & \specialcell[c]{$\El(3,\alpha,1)$ \\ $\oplus\,\El(3,-\alpha,1)\oplus(1)$} \\ [15pt]
\hline \\
$\bJ(7)$ & $\El(6,\alpha_1, 1)\oplus(-1)$ \\ [10pt]
$(\varepsilon\bJ(3), \varepsilon^{-1}\bJ(3),1)$ & $\El(6,\alpha, 1)\oplus(-1)$ \\ [10pt]
$(z\bJ(2), z^{-1}\bJ(2), z^2, z^{-2},1)$ & $\El(6,\alpha, 1)\oplus(-1)$ \\ [10pt]
$(x\bJ(2),x^{-1}\bJ(2), \bJ(3))$ & $\El(6,\alpha, 1)\oplus(-1)$ \\ [10pt]
$(x,y,xy,(xy)^{-1},y^{-1},x^{-1}, 1)$ & $\El(6,\alpha, 1)\oplus(-1)$ \\ [10pt]
\end{tabular}
\end{center}
where $\alpha\in \CC^\times$. However it is easy to see that the last five connections are hypergeometric connections (their Euler characteristic on $\Gm$ is $-1$) and that the first in the list is the pull-back of a hypergeometric connection. Therefore Theorem $1$ of \cite{SY19} can be applied to compute their Hodge numbers. \par
\subsection{Outlook}
In \cite{FSY18} the authors use irregular Hodge theory to prove meromorphic continuation for $L$-functions associated to symmetric power moments of Kloosterman sums. This is done by computing the Hodge numbers of symmetric powers of Kloosterman connections which turn out to be zero or one. The same property for Hodge numbers is true for the $G_2$-connection $\cE_3$ in the case $b>a$ and $b \neq 3/4$ in which the irregular Hodge filtration is of maximal length with every graded quotient being one dimensional. \par
Already in the tame case Dettweiler and Sabbah \cite{DS13} use explicit computations of Hodge numbers to prove potential automorphy for a family of Galois representations attached to a tamely ramified rigid $G_2$-local system (which implies meromorphic continuation for the associated $L$-function). Both of these approaches rely on a potential automorphy criterion of Patrikis and Taylor \cite{PT15}. \par 
In the future we hope to relate the rigid irregular $G_2$-connections above to exponential motives in a similar fashion as done in \cite{FSY18} for Kloosterman connections. Since the irregular $G_2$-connection $\calE_3$ arises from a rigid local systems via pull-back and Fourier transform one may hope that one can furthermore relate these exponential motives to classical motives and finally prove a similar result for $L$-functions attached to those - using the maximality of the Hodge filtration. 
%

\subsection*{Acknowledgement} Studying the irregular Hodge filtration of rigid connections was suggested to us by Claude Sabbah and we thank him for that. We wish to thank him, Javier Fresán and Jeng-Daw Yu for helpful conversations on irregular Hodge filtrations. In addition we wish to thank Claude Sabbah and Michael Dettweiler for comments that helped improve a preliminary version of this article.

\subsection{Strategy} The crucial point for computing the irregular Hodge filtration is a certain stability that was observed first by Katz in \cite[Section 6.2.]{Ka90} for hypergeometric differential equations. Let $\calH$ be a confluent hypergeometric connection of type $(n,m), n>m$ on $\Gm$. Let $d=n-m$ and denote by $[d]$ the $d$-fold covering of $\PP^1$. For sufficiently generic parameters there exists a regular singular hypergeometric connection $\calH'$ such that the Fourier transform of $[d]^*\calH$ is isomorphic to $[d]^*\calH'$. This property was used in \cite{SY19} by Sabbah and Yu to compute irregular Hodge filtrations for hypergeometric connections. \par
We will prove a similar stability for the above mentioned $G_2$-systems and adopt the strategy of Sabbah and Yu to compute the Hodge filtrations. More precisely we proceed in the following steps. \par
First we construct rigid local systems $\calL$ for which the Fourier transform of a pullback along $[k]$ for some positive integer $k$ will give the rigid irregular connections $\calE_i$. In the case of hypergeometrics in \cite{SY19} the authors used a result of Fedorov \cite{Fe18} who computed the Hodge data for hypergeometric local systems on $\PP^1 \setminus \{0,1,\infty\}$. In our case no such result is available. We therefore have to explicitly compute the Hodge data of $[k]^*\calL$ using results from \cite{DS13} and \cite{DR19} and a result on pullbacks of variations of Hodge structures in \cite[Section 4]{SY19}. The construction of $\calL$ and computation of the Hodge data is carried out in Section \ref{s: hodge data}.  \par
The next step is to prove that the aforementioned stability actually holds. We do this via an explicit computation with differential operators in Section \ref{s: operators}. The crucial point here is that we know an explicit algorithmic construction of $\calE_i$ in terms of the Katz-Arinkin algorithm \cite{Arinkin10}, i.e. via middle convolution and Fourier transform. \par
Finally we want to apply the stationary phase formula \cite[Section 5, (7)]{SY19} to compute the irregular Hodge filtrations. Unfortunately in all cases the local monodromy at infinity will have eigenvalue one and stationary phase is not directly applicable. We can avoid this by applying a suitable middle convolution to make all eigenvalues non-trivial. This is done directly and explicitly in terms of differential operators. \par
The Fourier transform will transform the middle convolution into a twist with a rank one local system. By \cite[Lemma 1]{SY19} such a twist will only change the Hodge filtration by a global shift and we can conclude.

\subsection{Notations and preliminary results} Here we recall some notation from \cite{DS13}. 
\subsubsection{Local Hodge data} Let $\Delta$ be a disc and $j:\Delta^*\hookrightarrow \Delta$ the open punctured disc. 
Let $(V, F^\bullet V, \nabla, k)$ be a variation of polarized complex Hodge structure on $\Delta^*$ defined as in \cite{DS13}, Section~2.1. We will often denote this only by $V$. \par 
For any real number $a$ there is an extension $V^a$ of $V$ to $\Delta$ called the Deligne canonical lattice. It is characterised by the property that the eigenvalues of the residue of $\nabla$ lie in $[a,a+1)$. Letting $V^{>a}=\bigcup_{a'>a} V^{a'}$ the residue of $V^{>a}$ has eigenvalues in $(a,a+1]$. For $a\in [0,1)$ and $\lambda=\exp(-2\pi i a)$ we define the space of nearby cycles
\[\psi_a(V):=V^a/V^{>a}.\] 
This space comes equipped with the nilpotent endomorphism $N=-(t\partial_t-a)$ which induces a monodromy filtration on it. For any $\ell \in \ZZ_{\ge 0}$ there is the space of primitive vectors $P_\ell\psi_a(V)$ with respect to the monodromy filtration. Its dimension is the number of Jordan blocks of size $\ell+1$ for the action of $N$ on $\psi_a(V)$. \par
Defining $F^pV^a:= j_*F^pV\cap V^a$ we get a Hodge filtration $F^p\psi_a(V):=F^pV^a / F^pV^{>a}$ on $\psi_a(V)$. This further induces a filtration on $P_\ell\psi_a(V)$ and we refer to  \cite{DS13}, Section~2.2 for details. \par
We define
the local  Hodge numbers
\[  \nu_{a,l}^p(V):=\dim \gr^p_F {\rm P}_l\psi_{a}(V), \]
\[ \nu_{a,{\rm prim}}^p(V):=\sum_{l\geq 0}\nu_{a,l}^p(V) \]
\[ h^p(V):=\nu^p(V):=\sum_{a\in [0,1)}\nu_{a}^p(V).\]
\subsubsection{Hodge data on $\AA^1\setminus S$} Now let $(V, F^\bullet V, \nabla, k)$ be a variation of polarized complex Hodge structure on $\AA^1 \setminus S$. To indicate the local Hodge data of $V$ at the point $x\in S\cup \infty$ we write $\nu_{x,a,l}^p(V)$ and so on. \par
As in \cite{DS13}, Section~2.3 we define global Hodge numbers 
\[ \delta^p(V)=\deg \gr^p_F V^0 \]
and we also make use of
the following further local Hodge numbers defined in \cite{DR19}, Section 2,
\begin{eqnarray*} \omega^p_{x}(V) &:=&\nu^p_{x}(V)-\nu^p_{x,0,\prim}(V)  \\
  \omega^p_{\neq \infty }(V) &:=&\sum_{x\in S \setminus \infty}\omega^p_{x}(V)  \\
     \omega^p_{}(V) &:=&\sum_{x\in S } \omega^p_{x}(V). \\
\end{eqnarray*}
By \cite[Corollary 8.1]{Sim90} any  rigid local system $\calL$ on $\PP^1\setminus S$ (for $S$ finite) underlies a variation of polarized complex Hodge structure $(V,F^\bullet B, \nabla, k)$ on $\PP^1\setminus S$
(provided their local monodromy has eigenvalues with absolute value one). A general result of Deligne \cite[Prop. 1.13]{De87} implies that any such local system underlies at most one such variation (up to a shift of the Hodge filtration). 
When working with  rigid local systems we will often abuse notation 
and write for example $\nu_{x,a,l}^p(\calL)$ instead of $ \nu_{x,a,l}^p(V)$ to denote the corresponding Hodge invariant of the associated variation of Hodge structure. For readability we may sometimes even drop the $V$ completely if it is clear from the context. \par
To compute the global Hodge numbers for Kummer pullbacks we make use of the following Lemma.

\begin{lemma}\label{lem: global hodge}
Let $f:\PP^1 \rightarrow \PP^1$ be given by $f(t)=t^k$, $k\in \ZZ_{\ge 1}$, and $V$ a variation of polarized complex Hodge structure on $\PP^1\setminus S$ for some finite set $S$. The pullback $f^*V$ is again a variation of polarized complex Hodge structure. Then
\[\delta^p(f^*V)=k\delta^p(V)+\sum_{b \in [0,1)} \lfloor k b \rfloor(\nu_{0, b}^p+\nu_{\infty,b}^p).\]
\end{lemma}
\begin{proof} Denote by $V^{-\mathbfit{j}/\mathbfit{k}}$ the extension to $\PP^1$ of $V$ given by the Deligne canonical lattices $V^{-{j}/{k}}$ at $0$ and at $\infty$ and by $V^0$ at the other singularities (see \cite{DS13}, Section~2.2). We have an exact sequence 
\[0 \rightarrow V^0 \rightarrow V^{-\mathbfit{j}/\mathbfit{k}} \rightarrow V^{-\mathbfit{j}/\mathbfit{k}} / V^0\rightarrow 0 \]
where 
\[V^{-\mathbfit{j}/\mathbfit{k}} / V^0\cong \bigoplus_{b\in [-j/k,0)} \psi_{b,t}(V)\oplus \psi_{b,t^{-1}}(V)\]
 and everything is compatible with the Hodge filtration. By \cite[Lemma 2]{SY19} we have
 \[(f^*V)^0=\bigoplus_{j=1}^{k-1} y^j\otimes V^{-\mathbfit{j}/\mathbfit{k}}\] 
 and hence 
 \begin{align*}\delta^p(f^*V)=\deg\gr_F^p(f^*V)^0&=\sum_{j=1}^{k-1}\left( \deg\gr_F^pV^0 +\sum_{b\in [-j/k,0)} \nu_{0,b}^p(V)+\nu_{\infty,b}^p(V)\right) \\
 &=k\delta^p(V)+\sum_{j=1}^{k-1}\sum_{b\in [1-j/k,1)} (\nu_{0,b}^p(V)+\nu_{\infty,b}^p(V)). 
 \end{align*}
 Let $b\in [0,1)$ such that $\nu_{0,b}^p(V)\neq 0$ and assume that $1-\frac{j}{k} \le b < 1-\frac{j-1}{k}$. This means that $k-j\le k b < k-j-1$ and hence $\lfloor k b\rfloor=k-j$. This is precisely the number of intervals $[1-j/k,1)$ in which $ b$ is contained. The same discussion applies to $\nu_{\infty, b}^p(V)\neq 0$ and hence 
 \[k\delta^p(V)+\sum_{j=1}^{k-1}\sum_{ b\in [1-j/k,1)} (\nu_{0, b}^p(V)+\nu_{\infty, b}^p(V))=k\delta^p(V)+\sum_{ b \in [0,1)} \lfloor k b \rfloor(\nu_{0, b}^p(V)+\nu_{\infty, b}^p(V)),\]
 proving the claim. 
\end{proof}

%

\section{Hodge data for rigid local systems} \label{s: hodge data}
In this section we construct rigid local systems $\calP_{13}, \calP_2$ and $\calP_4$ whose pullbacks are related to the pullbacks of $\calE_1, \calE_2, \calE_3$ and $\calE_4$ via Fourier transform. 
\begin{prop}\label{E1E3}
 Let $a,b \in [1/2,1)$ and $\alpha:=\exp(-2\pi ia),\beta:=\exp(-2\pi i b)$, where $\alpha \neq \beta^{\pm 1}$.
 There exists a regular singular connection $\calP_{13}$ on 
  $\PP^1\setminus \{ 0,1,4,\infty\}$ with symplectic monodromy group and
  local monodromy
 \begin{eqnarray*}
\begin{array}{cccc}
     0 & 1 & 4 & \infty \\
     \hline
     (\bJ(2),1,1) & (\alpha,\alpha^{-1},1,1) &  (\bJ(2),1,1) & (\beta,\beta,\beta^{-1},\beta^{-1})
    \end{array}, &&\quad \alpha\neq -1,\beta\neq -1,\\
  \begin{array}{cccc}
     0 & 1 & 4 & \infty \\
     \hline
     (\bJ(2),1,1) & (-\bJ(2),1,1) &  (\bJ(2),1,1) & (\beta,\beta,\beta^{-1},\beta^{-1})
    \end{array},&& \quad \alpha=-1,\beta\neq -1,\\  
    \begin{array}{cccc}
     0 & 1 & 4 & \infty \\
     \hline
     (\bJ(2),1,1) & (\alpha,\alpha^{-1},1,1) &  (\bJ(2),1,1) & (-\bJ(2),-\bJ(2))
    \end{array}, &&\quad \alpha\neq -1, \beta= -1 
    \end{eqnarray*}
 Furthermore $\calP_{13}$ has the following Hodge data
  \begin{eqnarray*} \begin{array}{ccccc}
     p & h^p   & \delta^p   & \omega^p & \nu^p_{\infty,b,1} \\
     \hline
     0 & 2 & -2 & 5 & 0  \\
     1 & 2 & -1 & 3 & 2
    \end{array},&&\quad b=1/2, \\
    \begin{array}{cccccc}
     p & h^p   & \delta^p   & \omega^p & \nu^p_{\infty,b,0} &  \nu^p_{\infty,1-b,0}\\
     \hline
     0 & 2 & -2 & 5 & 2 & 0 \\
     1 & 2 & -1 & 3 & 0 & 2
    \end{array},&&\quad 1/2<b< a, \quad    \\
     \begin{array}{cccccc}
     p & h^p   & \delta^p   & \omega^p & \nu^p_{\infty,b,0} &  \nu^p_{\infty,1-b,0}\\
     \hline
     0 & 2 & -2 & 5 & 1 & 1 \\
     1 & 2 & -1 & 3 & 1 & 1
    \end{array},&&\quad b > a.
\end{eqnarray*}
\end{prop}

\begin{proof}
Using the Katz algorithm \cite[Chapter~6]{Ka96} we construct $\calP_{13}$ via the following sequence of middle convolutions $\MC_\lambda$ and tensor products with  regular singular connections of rank 1 on 
  $\PP^1\setminus \{ 0,1,4,\infty\}$
 from a regular singular connection of rank 1 on 
  $\PP^1\setminus \{ 0,1,4,\infty\}$. 
 
 Let ${\mathcal M}$ be the regular singular connection of rank 1 on 
  $\PP^1\setminus \{ 0,1,4,\infty\}$
  with local monodromies
  \[  \begin{array}{ccccc}
     &0 & 1 & 4 & \infty \\
     \hline
    {\mathcal M} &  \alpha^{-1} & \alpha \beta & \alpha^{-1} & \alpha \beta^{-1}
    \end{array}.\]               
  Then
  \[ \MC_\beta (\MC_{\beta^{-1}\alpha} {\mathcal M} \otimes {\mathcal M}_2)={  \calP_{13}},\]
  where ${\mathcal M}_2$ is a regular singular connection of rank 1 on 
  $\PP^1\setminus \{ 0,1,4,\infty\}$
   with local monodromies
  \[  \begin{array}{ccccc}
     &0 & 1 & 4 & \infty \\
     \hline
    {\mathcal M}_2 &  1 & \alpha^{-1} \beta^{-1} & 1 & \alpha \beta^{} 
    \end{array}.\]     
%
   By the Katz algorithm \cite[Chapter~6]{Ka96} the local monodromy data is
   \[ \begin{array}{ccccc}
     &0 & 1 & 4 & \infty \\
     \hline
    {\mathcal M} &  \alpha^{-1} & \alpha \beta & \alpha^{-1} & \alpha \beta^{-1} \\
    \MC_{\beta^{-1}\alpha} {\mathcal M} & (1,\beta^{-1}) & (1,\alpha^2) & (1,\beta^{-1}) & (\alpha^{-1}\beta,\alpha^{-1}\beta) \\
    \MC_{\beta^{-1}\alpha} {\mathcal M} \otimes {\mathcal M}_2& (1,\beta^{-1}) & ( \alpha^{-1} \beta^{-1},\alpha \beta^{-1}) & (1,\beta^{-1}) & (\beta^2,\beta^2) \\
     \MC_\beta (\MC_{\beta^{-1}\alpha} {\mathcal M} \otimes {\mathcal M}_2)& (\bJ(2),1,1) & (\alpha,\alpha^{-1},1,1) &  (\bJ(2),1,1) & (\beta,\beta,\beta^{-1},\beta^{-1})
     \end{array}.\]
  Since the monodromy tuple of $\calP_{13}$ is linearly rigid and the eigenvalues of its local monodromies are invariant under taking
 inverses the monodromy group is selfdual.
 The transvection at $0$ implies that the monodromy group is contained in $\Sp_4$ and the Hodge length is even. 

  To compute the Hodge data of $ \calP_{13}$ one can now apply the Hodge theoretic version of the Katz algorithm by Dettweiler and Sabbah in \cite{DS13}
  and determine the Hodge data in each of the above steps.
  
  However in this special situation one can shortcut the computations.
  From  \cite[Proposition~2.3.3]{DS13} one concludes that the Hodge length of $ \calP_{13}$ is at most the number of used middle convolutions which is $3$.
  Since the Hodge length is even it has to be $2$.
  Thus $h^0(   \calP_{13})=h^1(   \calP_{13})=2$ by selfduality.
  Further, since 
  \[ \sum_{s \in \PP^1} \rk(T_s- {\rm id})-2\rk(   \calP_{13})=0,\]
  where $T_s$ denotes the local monodromy of $\calP_{13}$ at $s$,
  the parabolic cohomology group $H^1_{\rm par} (\calP_{13})$ vanishes
   cf. \cite[Section~2]{DR19} , i.e. $\calP_{13}$ is parabolically rigid.
   Therefore we get the following system of equations, cf. \cite[Proposition~2.7]{DR19},
  \[ \delta^{i-1}-\delta^i-h^{i-1}-h^{i} +\omega^{i-1}=0 .\]
  The selfduality and the definition of $\omega^p$ give $\omega^0=5$ and $\omega^1=3$, hence the claim on the $\delta^p$ follows.
  
  It remains to determine the nearby cycle data at $\infty$.
  If $b=1/2$ then the local monodromy at infinity already determines $\nu^1_{\infty,1/2,1}=2$. 
  
  Let $b>1/2$ and  $c \in \{a,1-a\}$ such that $\nu^0_{\infty,c,0}=1$.
  Let further $ { \mathcal L}_x, x \in (0,1),$ be the rank $1$ connection with local monodromy
  \[  \begin{array}{ccccc}
     &0 & 1 & 4 & \infty \\
     \hline
     { \mathcal L}_x &  1 & \exp(-2\pi ix)  & 1 & \exp(2\pi ix)
    \end{array}.\]
  Then 
  \[  \calP_{13} \otimes  { \mathcal L}_b,\quad  \calP_{13} \otimes {\mathcal L}_{1-b} \]
  are both parabolically rigid.
  Thus
  \[ -2=-h^0( \calP_{13})=-h^0( \calP_{13}\otimes {\mathcal L}_b)=\delta^0(  \calP_{13} \otimes {\mathcal L}_b) \]
  \[ -2=-h^0( \calP_{13})=\delta^0(  \calP_{13} \otimes {\mathcal L}_{1-b}).  \]
  Hence if $\nu^0_{\infty,d,0}( \calP_{13})=2, d \in \{ b,1-b\},$ then by \cite[Proposition]{DS13}
  \[ \delta^0(  \calP_{13} \otimes {\mathcal L}_b)=\delta^0( \calP_{13})-h^0( \calP_{13})+\lfloor c+b \rfloor +2 \lfloor d +1-b \rfloor \]
  and
  \[\delta^0(  \calP_{13} \otimes {\mathcal L}_{1-b})=\delta^0( \calP_{13})-h^0( \calP_{13})+\lfloor c+1-b \rfloor +2 \lfloor d +b \rfloor. \]
  Thus
 \[ d+1-b\geq 1,   \quad c+b<1,   \quad c+1-b <1,   \quad d+b\geq 1.\]
  This implies
  $ b=d $ and $ c<b<1-c=a.$
  On the other hand, if $\nu^0_{\infty,b,0}( \calP_{13})=\nu^1_{\infty,b,0}( \calP_{13})=1$ then by \cite[Proposition]{DS13}
  \[ \delta^0(  \calP_{13} \otimes {\mathcal L}_{1-b})=\delta^0( \calP_{13})-h^0( \calP_{13})+\lfloor c+1-b \rfloor +2 \]
   \[ \delta^0(  \calP_{13} \otimes {\mathcal L}_{b})=\delta^0( \calP_{13})-h^0( \calP_{13})+\lfloor c+b \rfloor +1. \]
   Thus
   $ c+1-b< 1$ and   $c+b\geq 1, $
   that is
   $ a= \max\{c,1-c\}< b. $
   This shows the claim.
\end{proof}

\begin{prop} \label{E1E3conv}
Let $\calP_{13}^2:=[2]^*  \calP_{13}$
 and $\calP':= \MC_\chi\calP_{13}^2$ for $\chi=\exp(-2\pi i \mu ), \mu \in [0,1), 1-\mu \sim 0$ by which we mean that $\mu$ is generic and sufficiently close to $1$.
 Then ${\calP'}$ has the following Hodge data
 \begin{eqnarray*}
   \begin{array}{ccccc}
     p & h^p   &  \nu^p_{\infty,1-\mu,2} &   \nu^p_{\infty,1-\mu,0}\\
      \hline
     0 & 2 &  0 & 0   \\
     1 & 3 & 0 & 1 \\
     2 & 2 & 2 & 0
    \end{array},&&\quad b=1/2, \\
   \begin{array}{cccccc}
     p & h^p     & \nu^p_{\infty,2b-\mu,0} &  \nu^p_{\infty,2(1-b)+1-\mu,0}&\nu^p_{\infty,1-\mu,0}\\
      \hline
     0 & 2  & 2 & 0 & 0 \\
     1 & 5  & 0 & 2 & 3 
    \end{array},&&\quad 1/2<b< a, b\neq 3/4 \\
    \begin{array}{cccccc}
     p & h^p     & \nu^p_{\infty,2b-\mu,0} &\nu^p_{\infty,1-\mu,0}\\
      \hline
     0 & 2  & 2  & 0 \\
     1 & 5  & 2 & 3 
    \end{array},&&\quad 3/4=b< a, \\    
    \begin{array}{cccccc}
     p & h^p   &  \nu^p_{\infty,2b-\mu,0} &  \nu^p_{\infty,2(1-b)+1-\mu,0} &  \nu^p_{\infty,1-\mu,0} \\
      \hline  
     0 & 3 & 1 &1 & 1 & \\
     1 & 3 & 1 & 1 & 1 &  \\
     2 & 1 & 0 & 0 & 1
    \end{array},&& \quad b > a, b\neq 3/4 \\
     \begin{array}{cccccc}
     p & h^p   &  \nu^p_{\infty,2b-\mu,0}&  \nu^p_{\infty,1-\mu,0} \\
      \hline  
     0 & 3 & 2  & 1 & \\
     1 & 3 & 2  & 1 &  \\
     2 & 1 & 0  & 1
    \end{array},&&\quad 3/4=b > a.
 \end{eqnarray*}
\end{prop}

\begin{proof}
One has the following Hodge data for ${  \calP}_{13}^2$
 \begin{eqnarray*}
   \begin{array}{cccccc}
     p & h^p   & \delta^p   & \omega^p  & \omega^p_{\neq \infty}   & \nu^p_{\infty,0,1}   \\
      \hline
     0 & 2 & -2 & 7&5 & 0  \\
     1 & 2 & 0 & 2&2 & 2
    \end{array},&& \quad b=1/2, \\
    \begin{array}{ccccccc}
     p & h^p   & \delta^p  & \omega^p   & \omega^p_{\neq \infty}   & \nu^p_{\infty,2b-1,0}   &  \nu^p_{\infty,2(1-b),0}  \\
      \hline
     0 & 2 & -2 & 7&5 & 2 & 0 \\
     1 & 2 & -2 & 4& 2 & 0 & 2
    \end{array},&& \quad 1/2<b< a, b\neq 3/4 \\
     \begin{array}{cccccc}
     p & h^p   & \delta^p   & \omega^p  & \omega^p_{\neq \infty}   & \nu^p_{\infty,2b-1,0}   \\
      \hline
     0 & 2 & -2 & 7&5 & 2  \\
     1 & 2 & -2 & 4& 2 & 2 
    \end{array},&& \quad b< a, b=3/4 \\
     \begin{array}{ccccccc}
     p & h^p & \delta^p & \omega^p& \omega^p_{\neq \infty}   & \nu^p_{\infty,2b-1,0}   &  \nu^p_{\infty,2(1-b),0}  \\
      \hline
     0 & 2 & -3 & 7&5 & 1 & 1 \\
     1 & 2 & -1 & 4& 2 & 1 & 1
    \end{array},&& \quad b > a \\
     \begin{array}{ccccccc}
     p & h^p & \delta^p & \omega^p& \omega^p_{\neq \infty}   & \nu^p_{\infty,2b-1,0}     \\
      \hline
     0 & 2 & -3 & 7&5 & 2  \\
     1 & 2 & -1 & 4& 2 & 2 
    \end{array},&& \quad 3/4=b > a.
\end{eqnarray*}
 The claim on $h^p  , \delta^p, \omega^p, \omega^p_{\neq \infty}, \nu^p$ follows from the above propositions, \cite[Lemma~2]{SY19} and the local
 monodromy data of ${  \calP}_{13}^2$ .
 
 By \cite[Proposition~5.3 (i)]{DR19}
 \[ h^p({  \calP}')=\delta^{p-1}({  \calP}_{13}^2)-\delta^p({  \calP}_{13}^2)+\omega^{p-1}_{\neq \infty}({  \calP}_{13}^2).\]
 The claim on 
  $\nu^p_{\infty,1-\mu,0}$ follows from  \cite[Corollary~6.2]{DR19} and \cite[Proposition~2.7]{DR19}
  \[ \nu^p_{\infty,1-\mu,0}({  \calP}')=h^p(H^1_{\para}({  \calP}_{13}^2))=\]
 \[  \delta^{p-1}({  \calP}_{13}^2)-\delta^p({  \calP}_{13}^2)-h^{p-1}({  \calP}_{13}^2)-h^{p}({  \calP}_{13}^2)+\omega^{p-1}({  \calP}_{13}^2) \]
  and the claim on
  $ \nu^p_{\infty,2b-\mu,l}({  \calP}'), \nu^p_{\infty,2(1-b)+1-\mu,l}({  \calP}')$ from
  \cite[Proposition~5.3 (ii)]{DR19}.
\end{proof}

\begin{prop}\label{E2}
 There is a rigid regular singular connection ${\mathcal P}_2$ of rank $2$ with symplectic monodromy group and local monodromy
  \[ \begin{array}{ccccc}
     0 & 1& 4& \infty \\
     \hline
   \bJ(2) & (\alpha,\alpha^{-1}) & \bJ(2) & (1,1)
     \end{array}.
   \]
 Further, let ${\mathcal P}_2^2=[2]^*{\mathcal P}_2$.
 Then for the rank $7$ connection $\MC_{-1}{\mathcal P}_2^2$ we have the Hodge data
  \[ \begin{array}{cccccc}
     p & h^p &  \nu^p_{\infty,1/2,0} & \nu^p_{\infty,1/2,1} \\
     \hline
    0 & 2 & 1 & 0\\
    1& 3& 1& 1 \\
    2 & 2& 1&  1
     \end{array}
   \]
\end{prop}
  
\begin{proof}  
  The connection $\calP_2$ may be constructed using the Katz algorithm \cite[Chapter 6]{Ka96}. Since the monodromy tuple of ${  {\mathcal P}_2}$ is linearly rigid and the eigenvalues of its local monodromies are invariant under taking
 inverses the monodromy group is selfdual.
 The transvection at $0$ implies that the monodromy group is contained in $\Sp_2$ and the Hodge length is two.
  The local monodromy gives $\omega^0=3, \omega^1=1$ and $\nu^0_{\infty,0,0}=\nu^1_{\infty,0,0}=1$.
  From the parabolic rigidity of ${\mathcal P}_2$ we infer from \cite[Proposition~2.3.3]{DS13}
 \[ 0=-\delta^0-h^0, \quad 0=\delta^0-\delta^1-h^0-h^1-\omega^0.\]
  Hence we have the following Hodge data
 \[ \begin{array}{ccccc}
     p & h^p & \delta^p& \omega^p& \nu^p_{\infty,0,0} \\
     \hline
    0 & 1 & -1  &3 &  1 \\
    1 & 1 &  0   &1  & 1
     \end{array}
   \]
  The global Hodge data of ${\mathcal P}_2^2$ can be determined from Proposition~\ref{lem: global hodge}  which yields
  \[ \begin{array}{ccccc}
     p & h^p & \delta^p& \nu^p_{\infty,0,0} \\
     \hline
    0 & 1 & -2 & 1 \\
    1& 1& 0& 1
     \end{array}
   \]
  
  The  Hodge data of  ${\mathcal P}_2^2 \otimes \calM,$ 
  where $\calM$  is the regular singular connection of rank $1$ with local monodromy $-1$ both at $\infty$ and at $t$ for any $t \in \PP^1\setminus\{0,\pm 1,\pm 2,\infty\}$
  are
  \[ \begin{array}{ccccc}
     p & h^p & \delta^p &\omega^p \\
     \hline
    0 & 1 & -3 & 7\\
    1& 1& -1& 4
     \end{array}
   \]
   by  \cite[Proposition~2.3.2]{DS13}.
  Since the  Hodge numbers of the parabolic cohomology of ${\mathcal P}_2^2 \otimes \calM$ are the  Hodge numbers of $\MC_{-1} {\mathcal P}_2^2$,
  we get by \cite[Theorem~4.3]{DR19}
  \[h^p(\MC_{-1}{\mathcal P}_2^2)=\delta^{p-1}-\delta^p-h^{p-1}-h^p+\omega^{p-1}.\]
  Therefore the Hodge numbers $h^p, p=0,1,2,$ of $\MC_{-1}{\mathcal P}_2^2$ are $2,3,2$.

   Since the monodromy of $\MC_{-1}{\mathcal P}_2^2$ is orthogonal by \cite[Corollary~5.10]{DR00} and the local monodromy at infinity is $(-\bJ(2),-\bJ(2),-1,-1,-1)$
   by the Katz algorithm \cite[Chapter~6]{Ka96}
   we get $\nu^1_{\infty,1/2,1}=\nu^2_{\infty,1/2,1}=1$.
   Having $h^p=\nu^p_{\infty,1/2}$   one concludes  $\nu^p_{\infty,1/2,0}=1$.
 
\end{proof}

\begin{prop}\label{E4}
 There is a rigid regular singular connection ${\mathcal P}_4$ of rank $4$ with symplectic monodromy group and local monodromy
  \[ \begin{array}{ccccc}
     0 & 1& t^2 &(t+1)^2 & \infty \\
     \hline
   (\bJ(2),1,1) &  (\bJ(2),1,1)&  (\bJ(2),1,1) &  (\bJ(2),1,1) & (-\bJ(2),-1,-1)
     \end{array}
   \]
 Further, let ${\mathcal P}_4^2=[2]^*{\mathcal P}_4$.
 Then for the rank $7$ connection $\MC_{-1}{\mathcal P}_4^2$ we have the Hodge data
  \[ \begin{array}{cccccc}
     p & h^p &  \nu^p_{\infty,1/2,1} & \nu^p_{\infty,1/2,2} \\
     \hline
    0 & 2 & 0 & 0\\
    1& 3& 1& 0 \\
    2 & 2& 1&  1
     \end{array}
   \]
\end{prop}

\begin{proof}  
  Since 
  ${\mathcal P}_4=\MC_{-1} {\mathcal M},$
  where ${\mathcal M}$ is the regular singular connection with local monodromy 
\[
\begin{array}{c c c c c c} 
	& 0 & 1 & t^2 & (t+1)^2 &  \infty \\  
		\hline 
		 & -1 & -1 &-1 & -1 & 1
		
\end{array}
\]
  the monodromy tuple of ${  {\mathcal P}_4}$ is linearly rigid and the monodromy group is contained in $\Sp_4$ by \cite[Corollary~5.10]{DR00}. 
   Thus the monodromy group of $\MC_{-1}{\mathcal P}_4^2$ is orthogonal by \cite[Corollary~5.10]{DR00}.
    From  \cite[Proposition~2.3.3]{DS13} one concludes that the Hodge length of $  \MC_{-1}{\mathcal P}_4^2$ is at most the number of used middle convolutions which is $3$.
   The local monodromy at infinity is $(-\bJ(2),-\bJ(3),-\bJ(2))$ by \cite[Chapter 6]{Ka96}.
   This implies that the Hodge length is $3$ and 
    $\nu^2_{\infty,1/2,2}=1$.
   By selfduality the Hodge numbers $h^p, p=0,1,2,$ are $a,b,a$ with $2a+b=7$. 
   Thus the Hodge numbers are $2,3,2$ and $\nu^1_{\infty,1/2,1}=\nu^2_{\infty,1/2,1}=1$.
 
\end{proof}

\section{Operators and Fourier transform} \label{s: operators} 
In the following we construct explicit differential operators corresponding to the connections we consider. Using these operators we can relate pullbacks of the irregular rigid $G_2$-connections to pullbacks of rigid local systems. \par
We use the following notation. Let $\delta$ denote the operator on $\CC[x]$ that sends a polynomial $f$ to its derivative $\frac{d}{dx}f$. Denote by $D=\CC[x]\langle \delta \rangle$ the ring of differential operators and let $\vartheta=x\delta$. \par
Consider the Fourier transform
\[FT:\CC[x]\langle \delta \rangle \rightarrow \CC[x]\langle \delta \rangle\]
defined by $FT(x)=\delta$ and $FT(\delta)=-x$. For any holonomic $D$-module its Fourier transform is the pullback along the map $FT$. In particular for any $P\in \CC[x]\langle \delta \rangle$ the module $D/DP$ is holonomic and its Fourier transform is the holonomic module $D/DFT(P)$.

Any differential operator $P\in \CC[x]\langle \delta \rangle$ has a finite singular locus $S\subset \PP^1$. Assume that $S=S'\cup \{0,\infty\}$. We may consider $P$ as an element of the localization $\CC[x, x^{-1}, (x-s)^{-1}, s\in S']\langle \vartheta \rangle$. As such it determines a linear homogeneous differential equation on $\PP^1\setminus S$. This in turn determines a connection on the trivial vector bundle on $\PP^1\setminus S$. Given a connection $\calE$ on $\PP^1 \setminus S$ we say that $P$ is the operator for $\calE$ if the connection on $\PP^1 \setminus S$ determined by $P$ is isomorphic to the connection $\calE$.

\begin{prop}\label{OpEi}
 The operator
 \begin{eqnarray*} L &:= &L_0(\vartheta)+xL_1(\vartheta)+x^2L_2(\vartheta)+x^3L_3(\vartheta)\in \CC[x][\vartheta],
 \end{eqnarray*}
 where
 \begin{eqnarray*}
   L_0(\vartheta)&=& 4    \vartheta  (  \vartheta-1)  (  \vartheta+1)  (2    \vartheta+1-2  b)  (-2  b-1+2    \vartheta)  (2  b+1+2    \vartheta)  (2  b-1+2    \vartheta) \\
   L_1(\vartheta)&=& -12    \vartheta  (  \vartheta+1)  (2    \vartheta+1)  (2  b+1+2    \vartheta)  (2    \vartheta+1-2  b)  \\
   L_2(\vartheta)&=&(  \vartheta+1)  (36    \vartheta^2+12  a^2-16  b^2+72    \vartheta-12  a+43)  \\
   L_3(\vartheta)&=& -6-4    \vartheta 
 \end{eqnarray*}
 is the operator for
 \begin{enumerate}
  \item ${\mathcal E}_1$ if $b\in 1/2+\ZZ$. 
  \item ${\mathcal E}_2$  if  $b\in \ZZ$.
  \item  ${\mathcal E}_3$ if $b\in \RR \setminus \{\ZZ \cup 1/2+\ZZ\}$.
 \end{enumerate}

\end{prop}

\begin{proof}
 The construction of $L$ follows from the construction of the ${\mathcal E}_i, i=1,2,3, $ in \cite[Theorem~1.1]{Ja20} and
 the translation to operators, cf. \cite[Chapter~4]{BR2012}, where
 one also uses
  \[ FT(\vartheta)=(-1-\vartheta),\quad FT(x)=\delta,\quad  x^i\delta^i=(\vartheta-i+1) \cdots \vartheta,\quad (\vartheta-1) x=x \vartheta\]
  and 
  \[ x^{\deg_x L} [1/x]^*L=\sum_{i=0}^{\deg_x L} x^{\deg_x L-i} L_i(-\vartheta). \] 
 \end{proof}

\begin{prop}\label{OpE4}
 The operator
 \begin{eqnarray*}
L&=&  L_0(\vartheta)+xL_1(\vartheta)+x^2L_2(\vartheta)+x^3L_3(\vartheta) \in \CC[x][\vartheta],
\end{eqnarray*}
where
\begin{eqnarray*}
   L_0(\vartheta)&=& 128 \vartheta^2 (\vartheta-2)^2 (\vartheta-1)^3\\
    L_1(\vartheta)&=& -32 \vartheta^2 (t^2+t+1) (2 \vartheta-1) (\vartheta-1)^2\\
     L_2(\vartheta)&=& 8 (t^2+t+1)^2 \vartheta^3\\
      L_3(\vartheta)&=& -t^2 (t+1)^2 (2 \vartheta+1)
   \end{eqnarray*}
is the operator for ${\mathcal E}_4$.
\end{prop}

\begin{proof}
 The construction of $L$ follows from the construction of ${\mathcal E}_4$ in \cite[Theorem~1.1]{Ja20} and
 the translation to operators
 as in Proposition~\ref{OpEi}.
\end{proof}

\begin{prop}\label{P:b<>0}
 Let 
 \begin{eqnarray*}
   P &:= &P_0(\vartheta)+xP_1(\vartheta)+x^2P_2(\vartheta)+x^3P_3(\vartheta)\in \CC[x][\vartheta],  \quad b\not \in \ZZ,\\
   P_0(\vartheta)&=& -4096  \vartheta^2  (\vartheta-1)  (\vartheta+1) \\
   P_1(\vartheta)&=& 1024  \vartheta  (\vartheta+1)  (9  \vartheta^2+3  a^2-4  b^2+9  \vartheta-3  a+4)  \\
   P_2(\vartheta)&=& -6144  (\vartheta+1)^2  (\vartheta+1+b)  (\vartheta+1-b) \\
    P_3(\vartheta)&=&1024  (\vartheta+1+b)  (\vartheta+1-b)  (\vartheta+2+b)  (\vartheta+2-b),
 \end{eqnarray*}
 be the symplectic operator of degree $4$ with Riemann scheme
 \[ \begin{array}{ccccc}
     0 & 1& 4& \infty \\
     \hline
     -1 & 0 & 0& 2-b \\
     0 & 1 & 1&  1-b \\
     0 & a& 1&  1+b \\
     1 & 1-a &2 & 2+b
    \end{array}
\]
Then $P$ is the operator for $\calP_{13}$ in Proposition~\ref{E1E3}.
\end{prop}

%
%
%
%

\begin{proof} 
The construction of $P$ follows from the construction of ${\mathcal P}_{13}$ in Proposition~\ref{E1E3} and
 the translation to operators, cf. \cite[Chapter~4]{BR2012}.
 Hence the claim follows.
\end{proof}

\begin{prop}\label{P:b=0}
  Let
  \begin{eqnarray*}
   P &:= &P_0(\vartheta)+xP_1(\vartheta)+x^2P_2(\vartheta)+x^3P_3(\vartheta) \in \CC[x][\vartheta],\quad b\in \ZZ
  \end{eqnarray*}
\begin{eqnarray*}
   P_0(\vartheta)&=&-4  \vartheta^2\\  
   P_1(\vartheta)&=& 9  \vartheta^2+3  a^2+9  \vartheta-3  a+4 \\
   P_2(\vartheta)&=& -6  (\vartheta+1)^2\\
   P_3(\vartheta)&=& (\vartheta+2)  (\vartheta+1)
 \end{eqnarray*}
 with Riemann scheme
 \[ \begin{array}{ccccc}
     0 & 1& 4& \infty \\
     \hline
     0 & -a& 0& 1\\
     0 & a-1& 0 &2
     \end{array}
\]

Then $P$  is the operator for $\calP_2$ in Proposition~\ref{E2}.
\end{prop}

%
%
%
%

\begin{proof} 
The construction of $P$ follows from the construction of ${{\mathcal P}_2}$ in Proposition~\ref{E2} and
 the translation to operators, cf. \cite[Chapter~4]{BR2012}.
 Hence the claim follows.
\end{proof}

\begin{prop}\label{P:E4}
 Let
 \begin{eqnarray*}
P&=& P_0(\vartheta)+xP_1(\vartheta)+x^2P_2(\vartheta)+x^3P_3(\vartheta) \in \CC[x][\vartheta],
\end{eqnarray*}
\begin{eqnarray*}
   P_0(\vartheta)&=& -16 \vartheta^2 t^2 (t+1)^2 (\vartheta-1) (\vartheta+1)\\
    P_1(\vartheta)&=& 4 \vartheta (t^2+t+1)^2 (\vartheta+1) (2 \vartheta+1)^2\\
    P_2(\vartheta)&=&-8 (t^2+t+1) (2 \vartheta+3) (2 \vartheta+1) (\vartheta+1)^2 \\
    P_3(\vartheta)&=&(2 \vartheta+5) (2 \vartheta+1) (2 \vartheta+3)^2 
   \end{eqnarray*}
 be the symplectic operator of degree $4$ with Riemann scheme
 \[ \begin{array}{ccccc}
     0 & 1& t^2& (t+1)^2 &\infty \\
     \hline
     -1 & 0 & 0&0& 1/2 \\
     0 & 1 & 1& 1& 3/2 \\
     0 & 1& 1& 1& 3/2 \\
     1 & 2 &2 & 2& 5/2
    \end{array}
\]
%
%
%
%

Then $P$  is the operator corresponding to $\calP_4$ in Proposition~\ref{E4}.
\end{prop}

\begin{proof} 
The construction of $P$ follows from the construction of ${\mathcal P}_4$ in Proposition~\ref{E4} and
 the translation to operators, cf. \cite[Chapter~4]{BR2012}.
 Hence the claim follows.
\end{proof}

%
%
%
%
%
%
%


\begin{remark}
 In the following we make use of the following transformations of operators.
 Let
 \begin{eqnarray*}
  P &=& \sum_{i=0}^{\deg_x(P)} x^{i}P_i(\vartheta).
 \end{eqnarray*}
 Then
 \begin{eqnarray*}
 [2]^*P&=&\sum_{i=0}^{\deg_x(P)} x^{2i}P_i(\vartheta/2)
  \end{eqnarray*}
  and
  \begin{eqnarray*}
 FT(P)&=&\sum_{i=0}^{\deg_x(P)} \delta^{i} P_i(-1-\vartheta),  \\
x^{\deg_x P} FT(P)      &=& \sum x^{\deg_x(P)-i}  (\vartheta-i+1) \cdots \vartheta \cdot P_i(-1-\vartheta).
\end{eqnarray*} 
Here we have used that
  \[ FT(\vartheta)=(-1-\vartheta),\quad FT(x)=\delta,\quad  x^i\delta^i=(\vartheta-i+1) \cdots \vartheta,\quad (\vartheta-1) x=x \vartheta.\]
\end{remark}

\begin{cor} \label{MC P132}
Let $\calP':= \MC_\chi\calP_{13}^2$ for $\chi=\exp(-2\pi i \mu ), \mu \in [0,1), 1-\mu \sim 0$,
 be as in Proposition~\ref{E1E3conv}.
The corresponding operator is
 \begin{eqnarray*}
  P' &=& \sum_{i=0}^3 x^{2i} P_{2i}(\vartheta),
 \end{eqnarray*}
where
\begin{eqnarray*}
 P_0(\vartheta) &=& -4  \vartheta  (\vartheta-5)  (\vartheta-1)  (\vartheta-2)  (\vartheta-3)  (\vartheta-4)  (\vartheta-\mu) \\
 P_2(\vartheta) &=& \vartheta  (\vartheta-1)  (\vartheta-2)  (\vartheta-3)  (\vartheta-\mu+1)  (9  \vartheta^2-18  \vartheta  \mu+12  a^2-16  b^2+9  \mu^2+18  \vartheta-12  a-18  
\mu+16)\\
P_4(\vartheta) &=& -6  \vartheta  (\vartheta-1)  (\vartheta-\mu+3)  (\vartheta-\mu+2)  (\vartheta-\mu+1)  (-\mu+2+2  b+\vartheta)  (-\mu+2-2  b+\vartheta) \\
P_6(\vartheta) &=& (\vartheta-\mu+1)  (-\mu+5+\vartheta)  (\vartheta-\mu+3)  (-\mu+2+2  b+\vartheta)  (-
\mu+2-2  b+\vartheta)  (-\mu+4+2  b+\vartheta)  (-\mu+4-2  b+\vartheta).
\end{eqnarray*}
Its Riemann scheme is
\[ \begin{array}{cccccc}
    -2 & -1 & 0 & 1 &2 & \infty \\
    \hline
    \mu & \mu-a & \mu & \mu-a &  \mu & 2b+4-\mu \\
    5   & \mu+a-1 & 5 & \mu+a-1 & 5 & 2b+2-\mu \\
    4 & 4 & 4 & 4 & 4 & 5-\mu \\
    3 & 3 & 3 & 3 & 3 & 3-\mu \\
    2 & 2 & 2 & 2 & 2 & 1-\mu  \\
    1 & 1 & 1&  1 & 1 & -2b+2-\mu \\
    0 & 0 & 0 & 0 & 0 &  -2b+4-\mu  \\
   \end{array}
   \]
%
%
%
%
%

Moreover
\begin{eqnarray*}
 FT(P') &=& \vartheta  (\vartheta-1)  (\vartheta-2)  (\vartheta-3)  (\vartheta-4)  (\vartheta-5)  H, \\
    H &=& \sum_{i=0}^3 x^{2i} H_{2i}(\vartheta), 
\end{eqnarray*}
where
\begin{eqnarray*}
  H_0(\vartheta) &=& -(\mu-4+\vartheta)  (\mu+\vartheta)  (\mu-2+\vartheta)  (\mu-3-2  b+\vartheta)  (\mu-3+2  b+\vartheta)  (\mu-1-2  b+\vartheta)  (\mu-1+2  b+\vartheta) \\
  H_2 (\vartheta)&=& 6  (\mu-2+\vartheta)  (\mu+\vartheta)  (\mu-1+\vartheta)  (\mu-1-2  b+\vartheta)  (\mu-1+2  b+\vartheta) \\
  H_4(\vartheta) &=& -(\mu+\vartheta)  (9  \vartheta^2+18  \vartheta  \mu+12  a^2-16  b^2+9  \mu^2-12  a+7) \\
  H_6 (\vartheta)&=& 4  (\vartheta+\mu+1).
\end{eqnarray*}
This is up to a twist with $x^{\mu+1}$ the operator
$ [2^*]P $ in Proposition~\ref{OpEi}.

\end{cor}

\begin{proof}
We start with the operator $P$ in Proposition~\ref{P:b<>0}.
The construction of $P'$ follows from the construction of ${\mathcal P}'$ in Proposition~\ref{E1E3} and
 the translation to operators, cf. \cite[Chapter~4]{BR2012}.
 The claim follows by a straightforward calculation using the above remark.
\end{proof}

\begin{cor} \label{MC P22}
Let $\calP':= \MC_{-1}\calP_{2}^2$  be as in Proposition~\ref{E2}.
Then the corresponding operator is
 \begin{eqnarray*}
  P' &=& \sum_{i=0}^3 x^{2i} P_{2i}(\vartheta),
 \end{eqnarray*}
where
\begin{eqnarray*}
P_0(\vartheta) & =&  256    \vartheta  (  \vartheta-1)  (  \vartheta-2)  (  \vartheta-3)  (  \vartheta-4)  (  \vartheta-5)  (2    \vartheta-1)\\
P_2(\vartheta) & =&  -16    \vartheta  (  \vartheta-1)  (  \vartheta-2)  (  \vartheta-3)  (2    \vartheta+1)  (36    \vartheta^2+48  a^2+36    \vartheta-48  a+37) \\
P_4(\vartheta) & =&  24    \vartheta  (2    \vartheta+5)
  (  \vartheta-1)  (2    \vartheta+1)  (2    \vartheta+3)^3 \\
P_6(\vartheta) & =&  -(2    \vartheta+5)  (2    \vartheta+1)  (2    \vartheta+9)  (2    \vartheta+7)^2  (2    \vartheta+3)^2.
\end{eqnarray*}
%
%
%
%
%
%
%
%
%
%
%
Its Riemann scheme is
\[ \begin{array}{cccccc}
    -2 & -1 & 0 & 1 &2 & \infty \\
    \hline
    1/2 & 1/2+a & 1/2 & 1/2+a &  1/2 & 9/2 \\
    5   & -1/2-a & 5 & -1/2-a & 5&7/2 \\
    4 & 4 & 4 & 4 & 4 & 7/2 \\
    3 & 3 & 3 & 3 & 3 & 5/2 \\
    2 & 2 & 2 & 2 & 2 & 3/2  \\
    1 & 1 & 1&  1 & 1 & 3/2 \\
    0 & 0 & 0 & 0 & 0 &  1/2 \\
   \end{array}
   \]
%
%
Thus
$FT(P')$ is
\begin{eqnarray*}
 FT(P') &=& \vartheta  (\vartheta-1)  (\vartheta-2)  (\vartheta-3)  (\vartheta-4)  (\vartheta-5)  H, \\
    H &=& \sum_{i=0}^3 x^{2i} H_{2i}(\vartheta), 
\end{eqnarray*}
where
\begin{eqnarray*}
 H_0(\vartheta) &=& (2    \vartheta-3)  (2    \vartheta+1)  (2    \vartheta-7)  (2    \vartheta-5)^2  (2    \vartheta-1)^2 \\
 H_2(\vartheta) &=&-24  (2    \vartheta-3)  (2    \vartheta+1)  (2    \vartheta-1)^3\\
 H_4(\vartheta) &=&16  (2    \vartheta+1)  (36    \vartheta^2+48  a^2+36    \vartheta-48  a+37)\\
 H_6(\vartheta)  &=& -256  (2    \vartheta+3).
\end{eqnarray*}
This is up to a twist with $x^{3/2}$ the operator
$ [2^*]P $ in Proposition~\ref{OpEi}.
%
%
%
%
%
\end{cor}

\begin{proof}
 The proof is similar to the above one.
\end{proof}

%
%
%
%
%
%

\begin{cor} \label{MC P42}
Let $\calP':= \MC_{-1}\calP_{4}^2$  be as in Proposition~\ref{E4}.
Then the corresponding operator is
 \begin{eqnarray*}
  P' &=& \sum_{i=0}^3 x^{2i} P_{2i}(\vartheta),
 \end{eqnarray*}
where
\begin{eqnarray*}
 P_0(\vartheta) & =& -64 \vartheta t^2 (t+1)^2 (\vartheta-1) (\vartheta-2) (\vartheta-3) (\vartheta-4) (\vartheta-5) (2 \vartheta-5) \\
 P_2(\vartheta) & =& 16 \vartheta (t^2+t+1)^2 (\vartheta-1) (\vartheta-2) (\vartheta-3) (2 \vartheta-3)^3 \\
 P_4(\vartheta) & =& -8 \vartheta (t^2+t+1) (2 \vartheta-1) (\vartheta-1) (2 \vartheta+1)^2 (2 \vartheta-3)^2 \\
 P_6(\vartheta) & =&   (2 \vartheta+5)^2 (2 \vartheta-3)^2 (2 \vartheta+1)^3\\
\end{eqnarray*}

Its Riemann scheme is
\[ \begin{array}{cccccc}
    \pm (t+1) & \pm t & \pm 1 & 0  & \infty \\
    \hline
     5 & 5 & 5 & 5 & 5/2 \\
      4 & 4 & 4 & 4 & 5/2 \\
       3 & 3 & 3 & 3 & 1/2 \\
        5/2 & 5/2 & 5/2 & 5/2 & 1/2 \\
        2 & 2 & 2 & 2 & 1/2 \\
        1 & 1 & 1 & 1 & -3/2 \\
         0 & 0 & 0 & 0 & -3/2 \\
    \end{array}
\]    

Thus
$FT(P')$ is
\begin{eqnarray*}
 FT(P') &=& \vartheta  (\vartheta-1)  (\vartheta-2)  (\vartheta-3)  (\vartheta-4)  (\vartheta-5)  H, \\
    H &=& \sum_{i=0}^3 x^{2i} H_{2i}(\vartheta), 
\end{eqnarray*}
where
\begin{eqnarray*}
 H_0(\vartheta) &=& -(2 \vartheta+5)^2 (2 \vartheta-3)^2 (2 \vartheta+1)^3 \\
 H_2(\vartheta) &=& 8 (t^2+t+1) (3+2 \vartheta) (2 \vartheta+1)^2 (2 \vartheta+5)^2\\
 H_4(\vartheta) &=& -16 (t^2+t+1)^2 (2 \vartheta+5)^3\\
 H_6(\vartheta)  &=& 64 t^2 (t+1)^2 (7+2 \vartheta).
\end{eqnarray*}

This is up to a twist with $x^{5/2}$ the operator
$ [2^*]P $ in Proposition~\ref{OpE4}.
\end{cor}

\begin{proof}
 The proof is again similar to the above one.
\end{proof}

%
%
%
%
%
%
%
%
%
%
%
%
%


Denote by $\calD$ the sheaf of differential operators on $\AA^1$ and by $\calD_{\Gm}$ its restriction to $\Gm$. Let $M_{13}:=\calD_{\Gm}/\calD_{\Gm}P'$ for $P'$ the operator in Corollary \ref{MC P132}. Similarly define $M_2$ using Corollary \ref{MC P22} and $M_4$ using Corollary \ref{MC P42}. Denoting by $S$ the singular locus of $P'$ the restriction of $M_{13}$ to $\PP^1\setminus S$ is isomorphic  to $\MC_\chi(\calP_{13}^2)$ and similarly in the other cases.

\begin{cor} \label{remstabi} Let $j:\Gm\hookrightarrow \AA^1$ be the open inclusion and denote by $K_\chi$ (and $K_{-1}$) the Kummer sheaves with monodromy $\chi$ at $0$ (and $-1$ at $0$). We have 
\begin{align*}
FT(j_{!*}M_{13})&\cong j_{!*}([2]^*\calE_1\otimes K_\chi), b\in1/2+\ZZ,\\
FT(j_{!*}M_{13})&\cong j_{!*}([2]^*\calE_3\otimes K_\chi), b\in \RR \setminus \{\ZZ \cup 1/2+\ZZ\} \\
FT(j_{!*}M_2) &\cong j_{!*}([2]^*\calE_2\otimes K_{-1}), \\
FT(j_{!*}M_4) &\cong j_{!*}([2]^*\calE_4\otimes K_{-1}).
\end{align*}
\end{cor}
\begin{proof} Again the proof is similar to the proof of Theorem 6.2.1 in \cite{Ka90}. We will only prove the first case as the others are similar.  None of the local exponents of $P'$ are contained in $\ZZ_{<0}$. We may therefore apply Lemma 2.9.5 in \cite{Ka90} to find that 
\[\calD/\calD P' \cong j_!\calD_{\Gm}/\calD_{\Gm}P'.\]
Therefore we have a short exact sequence
\[0\rightarrow \delta_0 \otimes \textup{Soln}_0^\vee \rightarrow \calD/\calD P' \rightarrow j_{!*} \calD_{\Gm}/\calD_{\Gm}P'\rightarrow 0 \]
where $\textup{Soln}_0^\vee$ denotes dual space of the formal local solutions at $0$ of $\calD/\calD P$. Applying Fourier transform we obtain the short exact sequence
\[0\rightarrow FT( \delta_0 \otimes \textup{Soln}_0^\vee) \rightarrow \calD/\calD FT(P')   \rightarrow  FT(j_{!*} \calD_{\Gm}/\calD_{\Gm}P')\rightarrow 0. \]
By Corollary \ref{MC P132} we have $FT(P')=\vartheta  (\vartheta-1)  (\vartheta-2)  (\vartheta-3)  (\vartheta-4)  (\vartheta-5)  H$. Note that 
\[FT( \delta_0 \otimes \textup{Soln}_0^\vee) = \calD / \calD (\vartheta  (\vartheta-1)  (\vartheta-2)  (\vartheta-3)  (\vartheta-4)  (\vartheta-5)).\]
Therefore  we have
\[FT(j_{!*} \calD_{\Gm}/\calD_{\Gm}P')\cong \calD / \calD H. \]
Recall from Corollary \ref{MC P132} that $H=\sum_{i=0}^3 x^{2i} H_{2i}(\theta)$. Then $H_0(X)$ considered as a polynomial in $X$ is the indicial polynomial of $H$ at $0$. Since it has no integer roots we may apply  Lemma 2.9.4 in \cite{Ka90} to conclude that 
\[\calD/ \calD H \cong j_{!*}j^* \calD_{\Gm}/\calD_{\Gm} H \cong j_{!*}([2]^*\calE_1\otimes K_\chi). \]

\end{proof}
Combining Corollary \ref{remstabi} with the computations of Hodge data for $\MC_\chi(\calP_{13}^2)$, $\MC_\chi(\calP_{2}^2)$ and $\MC_\chi(\calP_{4}^2)$ we now compute the irregular Hodge filtration of the rigid irregular $G_2$-connections. 
\begin{proof}[{Proof of Theorem \ref{thm: irreg hodge}}]By \cite[Theorem 0.3]{Sa18} pullback by a smooth morphism does not change the ranks and jumping indices of the irregular Hodge filtration. Additionally by \cite[Lemma 2.3]{SY19} twisting with a local system of rank one only changes the indices of the Hodge filtration by a global shift. We can therefore compute the irregular Hodge filtration of $\calE_{1}$ by computing it for the twist of $[2]^*\calE_{11}$. Note that we have arranged the local system $\MC_\chi(\calP_{13}^2)$  (where $b\in 1/2+\ZZ$) such that the local monodromy at $\infty$ does not have an eigenvalue equal to $1$. We can therefore apply the stationary phase formula \cite[Section 5, (7)]{SY19} to conclude that 
\[\dim \gr_{F_\textup{irr}}^{\alpha+p} (\calE_1)=\dim \gr_{F_\textup{irr}}^{\alpha+p} ([2]^*\calE_1\otimes K_\chi))= \nu_{\infty, \alpha}^p(\MC_\chi(\calP_{13}^2))= \sum_{\ell \ge 0} \sum_{k=0}^\ell \nu_{\infty, \alpha, \ell}^{p+k}(\MC_\chi(\calP_{13}^2)). \]
The other cases work the same and a straightforward calculation proves Theorem \ref{thm: irreg hodge}. 
\end{proof}

\bibliographystyle{alpha}
\bibliography{mybib}

\end{document}